\documentclass[11pt]{amsart}

\usepackage{geometry}   
\usepackage{graphicx}
\usepackage{amsmath}
\usepackage{amssymb}
\usepackage{amsthm}
\usepackage{epstopdf}
\usepackage{lscape}
\usepackage{array}
\usepackage[all,cmtip]{xy}
\usepackage{tikz}
\usetikzlibrary{matrix, arrows, decorations.pathreplacing}
\theoremstyle{theorem}
\newtheorem{thm}{Theorem}[section]
\newtheorem{lem}[thm]{Lemma}
\newtheorem{prop}[thm]{Proposition}

\theoremstyle{definition}

\newtheorem{rmk}[thm]{Remark}
\newtheorem{defn}[thm]{Definition}

\newcommand{\sI}{\mathcal{I}}
\newcommand{\sO}{\mathcal{O}}

\newcommand{\sR}{\mathcal{R}}

\newcommand{\Aa}{\mathbb{A}}
\newcommand{\CC}{\mathbb{C}}

\newcommand{\PP}{\mathbb{P}}
\newcommand{\QQ}{\mathbb{Q}}
\newcommand{\ZZ}{\mathbb{Z}}

\newcommand{\bA}{\mathbf{A}}
\newcommand{\bD}{\mathbf{D}}
\newcommand{\bE}{\mathbf{E}}

\newcommand{\frakm}{\mathfrak{m}}

\DeclareMathOperator{\Proj}{Proj}
\DeclareMathOperator{\Spec}{Spec}
\DeclareMathOperator{\Bl}{Bl}

\title[Divisorial Extractions from Singular Curves in Smooth $3$-folds, I]{Divisorial Extractions from Singular \\ Curves in Smooth $3$-folds, I}
\author{Tom Ducat}
\address{Mathematics Institute, University of Warwick, Coventry CV4 7AL, UK}
\email{T.Ducat@warwick.ac.uk}
\date{\today}         
                                  
\begin{document}

\maketitle
\setcounter{secnumdepth}{2}
\setcounter{tocdepth}{1}
\addtolength{\parskip}{0.5ex}

\begin{abstract}
Consider a singular curve $\Gamma$ contained in a smooth $3$-fold $X$.  Assuming the general elephant conjecture, the general hypersurface section $\Gamma\subset S\subset X$ is Du Val. Under that assumption, this paper describes the construction of a divisorial extraction from $\Gamma$ by Kustin--Miller unprojection. Terminal extractions from $\Gamma\subset X$ are proved not to exist if $S$ is of type $D_{2k}, E_7$ or $E_8$ and are classified if $S$ is of type $A_1,A_2$ or $E_6$. The $A_n$ and $D_{2k+1}$ cases are considered in a further paper. 
\end{abstract}

\tableofcontents

\section*{Introduction}

Much of the birational geometry of terminal $3$-folds has been classified explicitly. For example there is a classification of terminal $3$-fold singularities by Mori and Reid \cite{ypg}, a classification of exceptional $3$-fold flips by Koll\'ar and Mori \cite{km92} and, more recently, some work done by Hacking, Tevelev and Urz\'ua \cite{htu} and Brown and Reid \cite{dip} to describe type $A$ flips. In the case of divisorial contractions much is known, but there is not currently a complete classification. 

The focus of this paper is the case of a singular curve $\Gamma$ contained in a smooth $3$-fold $X$. Assuming the general elephant conjecture, the general hypersurface section $\Gamma\subset S\subset X$ is Du Val. Under this assumption, \S2 gives a normal form for the equations of $\Gamma$ in $X$ along with an outline for the construction of a (specific) divisorial extraction from $\Gamma$. In \S\S3-4 cases are studied explicitly by considering the type of Du Val singularity of $S$. I prove that if the general hypersurface through $\Gamma$ is a type $D_{2k}$ or $E_7$ singularity then a terminal divisorial extraction from $\Gamma$ does not exist. If it is a type $E_6$ singularity then there is an explicit description of the curves $\Gamma$ for which a terminal divisorial extraction exists. 

As well as treating the $A_1$ and $A_2$ cases, the main result in this paper is the following:
\begin{thm}
Suppose that $P\in \Gamma\subset X$ is the germ of a non-lci curve singularity $P\in\Gamma$ inside a smooth $3$-fold $P\in X$. Suppose moreover that the general hypersurface section $\Gamma\subset S\subset X$ is Du Val. Then,

\begin{enumerate}
\item if $S$ is of type $D_{2k}$ or $E_7$ then a terminal extraction from $\Gamma\subset X$ does not exist,

\item if $S$ is of type $E_6$ then a terminal extraction exists only if $\Gamma\subset S$ is a curve whose birational transform in the minimal resolution of $S$ intersects the exceptional locus with multiplicity given by either
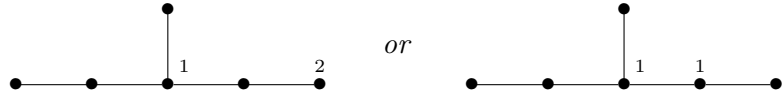
\begin{center} \begin{tikzpicture}
	\node at (1,1) {$\bullet$};
	\node at (2,1) {$\bullet$};
	\node[label={[label distance=-0.2cm]80:\tiny{$1$}}] at (3,1) {$\bullet$};
	\node at (4,1) {$\bullet$};
	\node[label={[label distance=-0.2cm]90:\tiny{$2$}}] at (5,1) {$\bullet$};
	\node at (3,2) {$\bullet$};
	\draw
       		(1,1)--(2.925,1)
       		(3.075,1)--(4.925,1)
       		(3,1.075)--(3,2);  
		
	\node at (6,1.5) {or};

	\node at (7,1) {$\bullet$};
	\node at (8,1) {$\bullet$};
	\node[label={[label distance=-0.2cm]80:\tiny{$1$}}] at (9,1) {$\bullet$};
	\node[label={[label distance=-0.2cm]90:\tiny{$1$}}] at (10,1) {$\bullet$};
	\node at (11,1) {$\bullet$};
	\node at (9,2) {$\bullet$};
	\draw
       		(7,1)--(8.925,1)
       		(9.075,1)--(9.925,1)
       		(10.075,1)--(10.925,1)
       		(9,1.075)--(9,2);  
\end{tikzpicture}  \end{center}
(an unlabelled node means multiplicity zero).
\end{enumerate}
\end{thm}
A more precise statement is given in Theorem \ref{excthm}.

\subsubsection{Acknowledgements}\

I would like to thank my PhD supervisor Miles Reid for suggesting this problem to me and for the help and guidance he has given me.

Part of this research was done whilst visiting Japan. I would like to thank Professor Yurijo Kawamata for inviting me and Stavros Papadakis for some helpful discussions whilst I was there.

\section{Preliminaries}

\subsection{Du Val singularities} \

The Du Val singularities are a very famous class of surface singularities. They can be defined in many different ways, a few of which are given here.

\begin{defn} 
Let $P\in S$ be the germ of a surface singularity. Then $P\in S$ is a \emph{Du Val singularity} if it is given, up to isomorphism, by one of the following equivalent conditions:
\begin{enumerate}
\item a hypersurface singularity $0\in V(f)\subset \Aa^3$, where $f$ is one of the equations of Table \ref{dvtab}, given by an ADE classification.
\item a quotient singularity $0\in \CC^2/G=\text{Spec }\CC[u,v]^G$, where $G$ is a finite subgroup of $\text{SL}(2,\CC)$.
\item a rational double point, i.e.\ the minimal resolution 
\[\mu \colon (E\subset \widetilde{S}) \to (P\in S)\] 
has exceptional locus $E=\bigcup E_i$ a tree of $-2$-curves with intersection graph given by the corresponding ADE Dynkin diagram. 
\item a canonical surface singularity. As $P\in S$ is a surface singularity this is equivalent to having a crepant resolution, i.e.\ $K_{\widetilde{S}}=\mu^*K_S$.
\item a simple hypersurface singularity, i.e.\ $0\in V(f)\subset \Aa^3$ such that there exist only finitely many ideals $I\subset\frakm$ with $f\in I^2$ (where $\frakm$ is the ideal of $0\in\Aa^3$).
\end{enumerate}
\label{dvdef}
\end{defn}

See for example \cite{duval} for details of (1)-(4) and \cite{yosh} for details of (5).
\begin{table}[htdp]
\caption{Types of Du Val singularities}
\label{dvtab}
\begin{center}
\begin{tabular}{cccc}
Type & Group $G$ & Equation $f$ & Dynkin diagram \\ \hline
$A_n$ & cyclic $\tfrac1r(1,-1)$  & $x^2 + y^2 + z^{n+1}$ & 
\begin{tikzpicture}[scale=0.8]
	\node[label={[label distance=-0.2cm]90:\tiny{1}}] at (1,1) {$\bullet$};
	\node[label={[label distance=-0.2cm]90:\tiny{1}}] at (2,1) {$\bullet$};
	\node at (3,1) {$\cdots$};
	\node[label={[label distance=-0.2cm]90:\tiny{1}}] at (4,1) {$\bullet$}; 
	\node[label={[label distance=-0.2cm]90:\tiny{1}}] at (5,1) {$\bullet$};  
	\node at (4,2) { };
	\node[below] at (3,1) {\tiny{($n$ nodes)}};
	\draw
       		(1,1)--(2.5,1);  
	\draw
       		(3.5,1)--(5,1);  
\end{tikzpicture} \\
$D_n$ & binary dihedral & $x^2 + y^2z + z^{n-1}$ & 
\begin{tikzpicture}[scale=0.8]
	\node[label={[label distance=-0.2cm]90:\tiny{1}}] at (1,1) {$\bullet$};
	\node[label={[label distance=-0.2cm]90:\tiny{2}}] at (2,1) {$\bullet$};
	\node at (3,1) {$\cdots$};
	\node[label={[label distance=-0.2cm]90:\tiny{2}}] at (4,1) {$\bullet$};
	\node[label={[label distance=-0.2cm]90:\tiny{1}}] at (5,1.5) {$\bullet$};
	\node[label={[label distance=-0.2cm]90:\tiny{1}}] at (5,0.5) {$\bullet$};
	\node[below] at (3,1) {\tiny{($n$ nodes)}};
	\draw
       		(1,1)--(2.5,1);  
	\draw
       		(3.5,1)--(4,1);  
	\draw
       		(4,1)--(5,1.5);  
	\draw
       		(4,1)--(5,0.5);  
\end{tikzpicture} \\ 
$E_6$ & binary tetrahedral & $x^2 + y^3 + z^4$ & 
\begin{tikzpicture}[scale=0.8]
	\node[label={[label distance=-0.2cm]90:\tiny{1}}] at (1,1) {$\bullet$};
	\node[label={[label distance=-0.2cm]90:\tiny{2}}] at (2,1) {$\bullet$};
	\node[label={[label distance=-0.2cm]80:\tiny{3}}] at (3,1) {$\bullet$};
	\node[label={[label distance=-0.2cm]90:\tiny{2}}] at (4,1) {$\bullet$};
	\node[label={[label distance=-0.2cm]90:\tiny{1}}] at (5,1) {$\bullet$};
	\node[label={[label distance=-0.2cm]90:\tiny{2}}] at (3,2) {$\bullet$};
	\draw
       		(1,1)--(5,1);  
	\draw
       		(3,1)--(3,2);  
\end{tikzpicture} \\ 
$E_7$ & binary octahedral & $x^2 + y^3 + yz^3$ & 
\begin{tikzpicture}[scale=0.8]
	\node[label={[label distance=-0.2cm]90:\tiny{1}}] at (1,1) {$\bullet$};
	\node[label={[label distance=-0.2cm]90:\tiny{2}}] at (2,1) {$\bullet$};
	\node[label={[label distance=-0.2cm]90:\tiny{3}}] at (3,1) {$\bullet$};
	\node[label={[label distance=-0.2cm]80:\tiny{4}}] at (4,1) {$\bullet$};
	\node[label={[label distance=-0.2cm]90:\tiny{3}}] at (5,1) {$\bullet$};
	\node[label={[label distance=-0.2cm]90:\tiny{2}}] at (6,1) {$\bullet$};
	\node[label={[label distance=-0.2cm]90:\tiny{2}}] at (4,2) {$\bullet$};
	\draw
       		(1,1)--(6,1);  
	\draw
       		(4,1)--(4,2);  
\end{tikzpicture}  \\ 
$E_8$ & binary isocahedral & $x^2 + y^3 + z^5$ &
\begin{tikzpicture}[scale=0.8]
	\node[label={[label distance=-0.2cm]90:\tiny{2}}] at (1,1) {$\bullet$};
	\node[label={[label distance=-0.2cm]90:\tiny{3}}] at (2,1) {$\bullet$};
	\node[label={[label distance=-0.2cm]90:\tiny{4}}] at (3,1) {$\bullet$};
	\node[label={[label distance=-0.2cm]90:\tiny{5}}] at (4,1) {$\bullet$};
	\node[label={[label distance=-0.2cm]80:\tiny{6}}] at (5,1) {$\bullet$};
	\node[label={[label distance=-0.2cm]90:\tiny{4}}] at (6,1) {$\bullet$};
	\node[label={[label distance=-0.2cm]90:\tiny{2}}] at (7,1) {$\bullet$};
	\node[label={[label distance=-0.2cm]90:\tiny{3}}] at (5,2) {$\bullet$};
	\draw
       		(1,1)--(7,1);  
	\draw
       		(5,1)--(5,2);  
\end{tikzpicture}  \\ 
\end{tabular}
\end{center}
\end{table}%

The numbers decorating the nodes of the Dynkin diagrams in Table \ref{dvtab} have several interpretations. For example, each node corresponds to the isomorphism class of a nontrivial irreducible representation of $G$ with dimension equal to the label. Another way these numbers arise is as the multiplicities of the $E_i$ in the \emph{fundamental cycle} $\Sigma\subset\widetilde{S}$ (that is, the unique minimal effective $1$-cycle such that $\Sigma \cdot E_i \leq 0$ for every component $E_i$ of $E$).

\subsection{Terminal $3$-fold singularities} \

One of the most useful lists at our disposal is Mori's list of $3$-fold terminal singularities. (See \cite{ypg} for a nice introduction.) Terminal singularities always exist in codim $\geq3$, so in the $3$-fold case they are all isolated points $P\in X$. They are classified by their index (the least $r\in\ZZ_{>0}$ such that $rD$ is Cartier, given any Weil divisor $D$ through $P\in X$). 

As shown by Reid, the index 1, or Gorenstein, singularities are \emph{compound Du Val} (cDV) singularities, i.e.\ isolated hypersurface singularities of the form
\[ \big(f(x,y,z) + tg(x,y,z,t) = 0\big) \subset \Aa^4_{x,y,z,t} \]
where $f$ is the equation of a Du Val singularity.

The other cases are the non-Gorenstein singularities. They can be described as cyclic quotients of cDV points by using the index to form a covering. For example, a singularity of type $cA/r$ denotes the quotient of a type $cA$ singularity
\[ \big( xy + f(z^r, t) = 0 \big) \subset \Aa^4_{x,y,z,t}  \: / \:  \tfrac1r(a,r-a,1,0) \]
where $\tfrac1r(a,r-a,1,0)$ denotes the $\ZZ/r\ZZ$ group action $(x,y,z,t) \mapsto (\epsilon^ax,\epsilon^{r-a}y,\epsilon z,t)$, for a primitive $r$th root of unity $\epsilon$. The general elephant of this singularity is given by an $r$-to-1 covering $A_{n-1}\to A_{rn-1}$. A full list can be found in \cite{km92}, p.\ 541.

\subsection{Divisorial contractions} \

\begin{defn}
A projective birational morphism $\sigma\colon Y\to X$ is called a \emph{divisorial contraction} if 
\begin{enumerate}
\item $X$ and $Y$ are quasiprojective $\QQ$-factorial (analytic or) algebraic varieties, 
\item there exists a unique prime divisor $E\subset Y$ such that $\Gamma=\sigma(E)$ has $\text{codim}_X\Gamma\geq2$, 
\item $\sigma$ is an isomorphism outside of $E$, 
\item $-K_Y$ is $\sigma$-ample and the relative Picard number is $\rho({Y/X})=1$.
\end{enumerate}
Given the curve $\Gamma\subset X$, we will also call any such $\sigma\colon Y\to X$ a \emph{divisorial extraction} of $\Gamma$. Moreover, if both $X$ and $Y$ have terminal singularities (so that this map belongs in the Mori category of terminal $3$-folds) then we call $\sigma$ a \emph{Mori contraction/extraction}. 
\end{defn}

Since the question of classifying divisorial contractions is local on $X$ we assume that $\sigma$ is a \emph{divisorial neighbourhood}, i.e.\ a map of $3$-fold germs
\[ \sigma \colon (Z\subset E\subset Y) \to (P\in\Gamma\subset X) \]
where $Z=\sigma^{-1}(P)$ is a (not necessarily irreducible) reduced complete curve. In practice $X$ is the germ of an affine variety over $\CC$ and it is assumed that we can make any analytic change of variables that needs to take place. In particular, as we are primarily interested in this paper with the case where $X$ is smooth, we can implicitly assume $(P\in X)\cong (0\in \Aa^3)$.

\subsubsection{Known results.}\

For $3$-folds, divisorial contractions fall into two cases:
\begin{enumerate}
\item $P=\Gamma$ is a point,
\item $P\in \Gamma$ is a curve.
\end{enumerate}

The first case has been studied intensively and is completely classified if $P\in X$ is a non-Gorenstein singularity. This follows from the work of a number of people---Corti, Kawakita, Hayakawa and Kawamata amongst others.

In either case, Mori and Cutkosky classify Mori contractions when $Y$ is Gorenstein. In particular, Cutkosky's result for a curve $\Gamma$ is the following.

\begin{thm}[Cutkosky \cite{cut}]\label{cutthm}
Suppose $\sigma\colon (E\subset Y)\to (\Gamma\in X)$ is a Mori contraction where $Y$ has at worst Gorenstein (i.e.\ index 1) singularities and $\Gamma$ is a curve. Then 
\begin{enumerate}
\item $\Gamma$ is a reduced, irreducible, local complete intersection curve in $X$,
\item $X$ is smooth along $\Gamma$, 
\item  $\sigma$ is isomorphic to the blowup of the ideal sheaf $\sI_{\Gamma/X}$, 
\item $Y$ only has $cA$ type singularities and 
\item a general hypersurface section $\Gamma\subset S$ is smooth.
\end{enumerate}
\end{thm}

Kawamata \cite{kawam} classifies the case when the point $P\in X$ is a terminal cyclic quotient singularity. In this case, there is a unique divisorial extraction given by a weighted blowup of the point $\Gamma=P$. In particular, if there exists a Mori extraction to a curve $\Gamma\subset X$, then $\Gamma$ cannot pass through any cyclic quotient point on $X$. 

Tziolas \cite{tz1,tz2,tz3,tz4} classifies terminal extractions when $P\in \Gamma\subset X$ is a smooth curve passing through a cDV point.

\subsection{The general elephant conjecture} \
\label{geconj}

Reid's general elephant conjecture states that, given a terminal contraction 
\[ \sigma\colon (E\subset Y) \to (\Gamma \subset X), \] 
then general anticanonical sections $T_Y \in|{-K}_Y|$ and $T_X=\sigma(T_Y) \in|{-K}_X|$ should have at worst Du Val singularities. Moreover, $\sigma\colon T_Y\to T_X$ should be a partial crepant resolution at $P\in T_X$.

This is proved by Koll\'ar and Mori \cite{km92} for \emph{extremal} neighbourhoods (i.e.\ ones where the central fibre $Z$ is irreducible). In almost all the examples constructed in this paper, $Z$ is reducible.

Note that the existence of a Du Val general elephant implies that the general hypersurface section $\Gamma\subset S\subset X$ is also Du Val. The construction of the divisorial extraction $\sigma\colon Y\to X$ (i.e.\ the equations and singularities of $Y$) depends upon the general section $S$ rather than the anticanonical section $T_X$. Therefore, through out this paper assume we are in the following local situation
\[P \in \Gamma \subset S \subset X \]
where $P\in\Gamma$ is a (non-lci) curve singularity, $P\in S$ is a general Du Val section and $(P\in X)\cong (0\in\Aa^3)$ is smooth.

\subsection{Uniqueness of contractions} \

The following Proposition appears in \cite{tz1} Proposition 1.2.
\begin{prop}
\label{uniq}
Suppose that $\sigma\colon Y\to X$ is a divisorial contraction that contracts a divisor $E$ to a curve $\Gamma$, that $X$ and $Y$ are normal and that $X$ has isolated singularities. Suppose further that $\sigma$ is the blowup over the generic point of $\Gamma$ in $X$ and that $-E$ is $\sigma$-ample. Then $\sigma\colon Y\to X$ are uniquely determined and isomorphic to 
\[ \Bl_\Gamma\colon\Proj_X \bigoplus_{n\geq 0} \sI^{[n]} \to X \]
where $\sI^{[n]}$ is the $n$th symbolic power of the ideal sheaf $\sI=\sI_{\Gamma/X}$, i.e.\ the blowup of the symbolic power algebra of $\sI$.
\end{prop}

\begin{proof}
Pick a relatively ample Cartier divisor class $D$ on $Y$ which must be a rational multiple of $\sO_Y(-E)$. Then
\[ Y = \Proj_X R(Y,D) \] 
and, up to truncation, this is the ring $R(Y,-E)$. 

Now the result follows from the claim that $\sigma_*\sO_Y(-nE)$ is the $n$th symbolic power of $\sI$. This is clear at the generic point of $\Gamma$, since we assume it is just the blowup there. Now $\sigma_*\sO_Y = \sO_X$ is normal, and $\sO_Y(-nE)\subset \sO_Y$ is the ideal of functions vanishing $n$ times on $E$ outside of $\sigma^{-1}( P)$. So $\sO_X/\sigma_*\sO_Y(-nE)$ has no associated primes other than $\Gamma$ and this proves the claim.
\end{proof}

\begin{rmk}
Suppose $\sigma\colon Y\to X$ is a terminal divisorial contraction. By Mori's result, $Y$ is the blowup over the generic point of $\Gamma$ and we are in the setting of the theorem. Therefore a terminal contraction is unique if it exists, although there may be many more canonical contractions to the same curve.  
\end{rmk}

From this result, it is also easy to see that Cutkosky's result, Theorem \ref{cutthm}, holds for divisorial extractions, as well as contractions. 

\begin{lem}
Suppose that $\Gamma$ is a local complete intersection curve in a $3$-fold $X$ and that $X$ is smooth along $\Gamma$. Then a Mori extraction exists iff\/ $\Gamma$ is reduced, irreducible and a general hypersurface section $\Gamma\subset S$ is smooth. 
\end{lem}

\begin{proof}
By Proposition \ref{uniq}, if a Mori extraction $\sigma\colon Y\to X$ exists then $\sigma$ is isomorphic to the blowup of the ideal sheaf $\sI_{\Gamma/X}$. As $\Gamma$ is lci then, locally at a point in $P\in\Gamma\subset X$, we have that $\Gamma$ is defined by two equations $f,g$. Hence $Y$ is given by
\[ Y = \{ f\eta - g\xi = 0 \} \subset X\times \PP^1_{(\eta:\xi)} \to X \]
If both $f,g\in\frakm^2$ then at any point $Q$ along the central fibre $Z=\sigma^{-1}(P)_{\text{red}}$ the equation defining $Y$ is contained in $\frakm_Q^2$. Therefore $Y$ is singular along $Z$ and hence not terminal. So at least one of $f,g$ is the equation of a smooth hypersurface, say $f\in\frakm\setminus \frakm^2$. Now $Y$ is smooth along $Z$ except for a possible $cA$ type singularity at the point $P_\xi\in Y$, where all variables except $\xi$ vanish.
\end{proof}

\subsection{Unprojection}\label{unproj} \

In this paper, divisorial contractions are constructed by Kustin--Miller unprojection. The general philosophy of unprojection is to start working explicitly with Gorenstein rings in low codimension, successively adjoining new variables with new equations. For more details on Type I unprojection and Tom \& Jerry, see e.g.\ \cite{par,bkr,kino}. 

All unprojections appearing in this paper are Gorenstein Type I unprojections. A point $Q\in Y$ on a $3$-fold is called a \emph{Type I centre} if we can factor the projection map $Q\in Y\dashrightarrow \Pi\subset Y'$ as 
\begin{center}
\begin{tikzpicture}
  \matrix (m) [matrix of math nodes,row sep=1em,column sep=2em,minimum width=2em]
  {
      & E\subset Z &  \\
     Q\in Y & & \Pi\subset Y' \\};
  \path[->] (m-1-2) edge node [above] {$\phi$} (m-2-1); 
  \path[->] (m-1-2) edge node [above] {$\psi$} (m-2-3); 
\end{tikzpicture}
\end{center}
where $\phi$ is a divisorial extraction from the point $Q\in Y$ with exceptional divisor $E\subset Z$, $\psi$ is a small birational anticanonical morphism mapping $E$ birationally to a divisor $\Pi\subset Y'$ such that both $Y'$ and $\Pi$ are projectively Gorenstein. The map $Y'\dashrightarrow Y$ is called the \emph{unprojection map}. The point is that under these conditions the Kustin--Miller unprojection of the divisor $\Pi\subset Y'$ (described in \cite{par} Theorem 1.5) reconstructs $Y$, so that $Y$ can be obtained by adjoining just one new variable $u$ to the graded ring defining $Y'$, with a systematic way of obtaining the equations involving $u$ (in practice it is usually easy to work them out by ad hoc methods).

Another important idea appearing in these calculations is the structure of Gorenstein rings in codimension 3. By a theorem of Buchsbaum and Eisenbud, the equations of such a ring can be written as the maximal Pfaffians of a skew-symmetric $(2k+1)\times (2k+1)$ matrix. In practice we can usually always take $5\times 5$ matrices.

Now suppose that $Y'$ is a $3$-fold in codimension 3, given by the maximal Pfaffians of a $5\times 5$ skew matrix $M$. \emph{Tom \& Jerry} are the names of two different restrictions on $M$ that are necessary for $Y'$ to contain a plane $\Pi$, defined by an ideal $I$. These are:
\begin{enumerate}
\item Tom$_i$---all entries of $M$ except the $i$th row and column belong to $I$,
\item Jer$_{ij}$---all entries of $M$ in the $i$th and $j$th rows and columns belong to $I$.
\end{enumerate}

The easiest way to understand all of this is to work through the example given in \S\ref{prokreid}, with a geometrical explanation given in Remark \ref{geom}.

\section{Curves in Du Val Singularities}

Let $\Gamma$ be a reduced and irreducible curve passing through a Du Val singularity $(P\in S)$. Consider $S$ as simultaneously being both the hypersurface singularity $0\in V(f)\subset \Aa^3$, as in Definition \ref{dvdef}(1), and the group quotient $\pi\colon\CC^2\to\CC^2/G$, as in Definition \ref{dvdef}(2). Write $S=\text{Spec }\sO_S$ where
\[ \sO_S = \sO_X/(f)=(\sO_{\CC^2})^G, \quad \sO_X=\CC[x,y,z], \quad \sO_{\CC^2}=\CC[u,v]. \]
The aim of this section is to describe the equations of $\Gamma\subset X\cong \Aa^3$ in terms of some data associated to the equation $f$ and the group $G$.

\subsection{A 1-dimensional representation of $G$} \

Consider $C := \pi^{-1} (\Gamma) \subset \CC^2$, the preimage of $\Gamma$ under the quotient map $\pi$. Then $C$ is a reduced (but possibly reducible) $G$-invariant curve giving a diagram
\begin{center}
\begin{tikzpicture}
  \matrix (m) [matrix of math nodes,row sep=1em,column sep=2em,minimum width=2em]
  {
     C & \CC^2_{u,v} \\
     \Gamma & S \\};
  \path[->] (m-1-1) edge (m-2-1); 
  \path[->] (m-1-2) edge node [right] {$\pi$} (m-2-2);
  \path[right hook->]  (m-1-1) edge (m-1-2);
  \path[right hook->]  (m-2-1) edge (m-2-2);
  \node [left=4pt] at (m-2-1) {$P\in$};
\end{tikzpicture}
\end{center}
As such, $C$ is defined by a single equation $V(\gamma)\subset\CC^2$ and this $\gamma(u,v)$ is called the \emph{orbifold equation} of $\Gamma$. As $C$ is $G$-invariant the equation $\gamma$ must be $G$-semi-invariant, so there is a 1-dimensional representation $\rho\colon G\to\CC^\times$ such that 
\[ {^g\gamma}(u,v) = \rho(g) \gamma(u,v), \quad\quad \forall g\in G. \]
Moreover, $\Gamma$ is a Cartier divisor (and hence lci in $X$) if and only if $\rho$ is the trivial representation. Let us restrict attention to nontrivial $\rho$.

As can be seen from Table \ref{dvtab}, there are $n$ such representations if $S$ is type $A_n$, three if type $D_n$, two if type $E_6$, one if type $E_7$ and none if type $E_8$. These possibilities are listed later on in Table \ref{mftab}.

\subsection{A matrix factorisation of $f$} \

As is well known from the McKay correspondence, the ring $\sO_{\CC^2}$ has a canonical decomposition as a direct sum of $\sO_S$-modules
\[ \sO_{\CC^2} = \bigoplus_{\rho\in \text{Irr}(G)} M_\rho \] 
where $M_\rho = V_\rho \otimes \text{Hom}(V_\rho, \sO_{\CC^2})^G$ and $\text{Irr}(G)$ is the set of irreducible $G$-representations $\rho\colon G\to\text{GL}(V_\rho)$. In particular if $\dim\rho=1$ then we see that $M_\rho$ is the unique irreducible summand of $\sO_{\CC^2}$ of $\rho$ semi-invariants
\[ M_\rho=\big\{h(u,v)\in\sO_{\CC^2} : {^gh}=\rho(g)h \big\}. \]
This is a rank 1 maximal Cohen-Macaulay $\sO_S$-module generated by two elements at $P$. 

As shown by Eisenbud \cite{eis}, such a module over the ring of a hypersurface singularity has a minimal free resolution which is 2-periodic, i.e.\ there is a resolution
\[ \begin{matrix} 
M_\rho & \leftarrow & \sO_S^{\oplus2} & \stackrel{\phi}{\longleftarrow} & \sO_S^{\oplus2} & \stackrel{\psi}{\longleftarrow} & \sO_S^{\oplus2} & \stackrel{\phi}{\longleftarrow} & \cdots
\end{matrix} \]
where $\phi$ and $\psi$ are matrices over $\sO_X$ satisfying 
\[ \phi\psi = \psi\phi = fI_2.\]
The pair of matrices $(\phi,\psi)$ is called a \emph{matrix factorisation} of $f$. In our case $\phi$ and $\psi$ are $2\times 2$ matrices. It is easy to see that $\det\phi=\det\psi=f$ and that $\psi$ is the adjugate matrix of $\phi$. Write $I(\phi)$ for the ideal of $\sO_X$ generated by the entries of $\phi$ (or equivalently $\psi$).

Write $\epsilon_k$ (resp.\ $\omega,i$) for a primitive $k$th (resp.\ 3rd, 4th) root of unity. In Table \ref{mftab} the possible representations $\rho$ of $G$ and the first matrix $\phi$ in a matrix factorisation of $M_\rho$, for some choice of $f$, are listed. These can be found, for instance, in \cite{kst} \S5.
\begin{table}[htdp]
\caption{1-dimensional representations of $G$} 
\label{mftab}
\begin{center}
\begin{tabular}{cccc}
Type & Presentation of $G$ & $\rho(r ), \big(\rho(s),\rho(t)\big)$ & $\phi$ \\ \hline \\
$\bA_{n}^j$ & $\left\langle r : r^{n+1}=e \right\rangle$ & $\epsilon_{n+1}^j$  & $\begin{pmatrix} x & y^j \\ y^{n+1-j} & z \end{pmatrix}$  \\ \\
$\bD_{n}^l$ & $\left\langle \begin{matrix} r,s,t : \\ r^{n-2}=s^2=t^2=rst \end{matrix} \right\rangle$ & $1,-1,-1$ &  $\begin{pmatrix} x & y^2 + z^{n-2} \\ z & x\end{pmatrix}$ \\  \\
$\bD_{2k}^r$ &   & $-1,1,-1$ & $\begin{pmatrix} x & yz + z^k \\ y & x\end{pmatrix}$ \\ \\
$\bD_{2k+1}^r$ &   & $-1,i,-i$ & $\begin{pmatrix} x & yz \\ y & x+z^k \end{pmatrix}$  \\ \\
$\bE_6$ & $\left\langle \begin{matrix} r,s,t : \\ r^2=s^3=t^3=rst \end{matrix} \right\rangle$ & $1,\omega,\omega^2$ & $\begin{pmatrix} x & y^2 \\ y & x + z^2 \end{pmatrix}$  \\ \\
$\bE_7$ & $\left\langle \begin{matrix} r,s,t : \\ r^2=s^3=t^4=rst \end{matrix} \right\rangle$ & $-1,1,-1$ & $\begin{pmatrix} x & y^2+z^3 \\ y & x \end{pmatrix}$  
\end{tabular}
\end{center}
\end{table}%

The notation $\bD_n^l$ refers to the case when $\rho$ is the 1-dimensional representation corresponding to the leftmost node in the $D_n$ Dynkin diagram (see Table \ref{dvtab}) and $\bD_n^r$ refers to one of the rightmost pair of nodes. Of course there are are actually two choices of representation we could take for each of the cases $\bD_{2k}^r,\bD_{2k+1}^r$ and $\bE_6$. However we treat each of them as only one case since there is an obvious symmetry of $S$ switching the two types of curve. Similarly for $\bA_n^j$ we can assume that $j\leq\tfrac{n+1}2$.

\subsection{Normal forms for $\Gamma\subset X$}\

\begin{lem} Suppose that we are given $P\in\Gamma\subset S\subset X$ as in \S\ref{geconj}. Let $\rho$ and $\phi$ be the representation of $G$ and matrix factorisation of $f$ associated to $\Gamma$. Then
\begin{enumerate}
\item  the equations of $\Gamma\subset X$ are given by the minors of a $2\times 3$ matrix
\[ \bigwedge^2 \begin{pmatrix} \: \phi & \begin{matrix} g \\ h \end{matrix} \end{pmatrix} = 0 \]
for some functions $g,h\in\sO_X$.
\item Suppose furthermore that $S$ is a \emph{general} Du Val section containing $\Gamma$. Then $g,h\in I(\phi)$.
\end{enumerate}
\label{Lem1}
\end{lem}

\begin{proof}
Suppose that $\rho$ is a 1-dimensional representation of $G$. Note that if $(\psi,\phi)$ is a matrix factorisation for $M_\rho$, the $\sO_S$-module of $\rho$ semi-invariants, then $(\phi,\psi)$ is a matrix factorisation for $M_{\rho'}$, where $\rho'$ is the representation $\rho'(g)=\rho(g)^{-1}$.

The resolution of the $\sO_{\CC^2}$-module $\sO_C = \sO_{\CC^2}/(\gamma)$
\[ \begin{matrix} 
\sO_C & \leftarrow & \sO_{\CC^2} & \stackrel{\gamma}{\longleftarrow} & \sO_{\CC^2} & \leftarrow & 0
\end{matrix} \]
decomposes as a resolution over $\sO_S$ to give a resolution of $\sO_\Gamma$
\[ \begin{matrix} 
\sO_\Gamma & \leftarrow & \sO_S & \stackrel{\gamma}{\longleftarrow} & M_{\rho'} & \leftarrow & 0.
\end{matrix} \]
Using the resolution of $M_{\rho'}$ we get 
\[ \begin{matrix} 
\sO_\Gamma & \leftarrow & \sO_S & \xleftarrow{(\xi_2 \: -\xi_1)} & \sO_S^{\oplus2} & \stackrel{\phi}{\longleftarrow} & \sO_S^{\oplus2} & \stackrel{\psi}{\longleftarrow} & \cdots,
\end{matrix} \]
where $\xi_1,\xi_2$ are the two equations defining $\Gamma\subset S$. Now write $\gamma=g\alpha + h\beta$ where $\alpha,\beta$ are the two generators of $M_{\rho}$. We can use the resolution of $\sO_S$ as an $\sO_X$-module to lift this to a complex over $\sO_X$ and strip off the initial exact part to get the resolution
\[ \begin{matrix} 
\sO_\Gamma & \leftarrow & \sO_X & \xleftarrow{(\xi_2 \: -\xi_1 \: \eta)} & \sO_X^{\oplus3} & 
\xleftarrow{ \tiny \begin{pmatrix} \phi \\ \: g \:\: h \: \end{pmatrix}} & \sO_X^{\oplus2} & \leftarrow & 0
\end{matrix} \]
(possibly modulo some unimportant minus signs). Therefore the equations of the curve $\Gamma\subset S\subset X$ are given as claimed in (1).

To prove Lemma \ref{Lem1}(2), recall the characterisation of Du Val singularities in Definition \ref{dvdef}(5) as simple singularities. Let $\eta=\det{\phi}$ and $\xi_1,\xi_2$ be the three equations of $\Gamma$. We have a $\CC^2$-family of hypersurface sections through $\Gamma$ given by
\[  H_{\lambda,\mu} = \big\{ h_{\lambda,\mu} := \eta+\lambda \xi_1 +\mu \xi_2=0 \big\}_{(\lambda,\mu)\in\CC^2} \]
and we are assuming that $\eta$ is general. As the general member $H_{\lambda,\mu}$ is Du Val there are a finite number of ideals $I\subset\frakm$ such that the general $h_{\lambda,\mu}\in I^2$. As the general section $\eta$ satisfies $\eta\in I(\phi)^2$ we have that $h_{\lambda,\mu}\in I(\phi)^2$ for general $\lambda,\mu$. Therefore $g,h\in I(\phi)$.
\end{proof}

\begin{rmk}
Whilst Lemma $\ref{Lem1}$ gives a necessary condition, $g,h\in I(\phi)$, for a general section of a curve $\Gamma$ to be of the same type as $S$, it is not normally a sufficient condition.
\end{rmk}

\subsection{The first unprojection}\
\label{1st_unproj}

Now $\Gamma$ is defined as the minors of a $2\times3$ matrix, where all the entries belong to an ideal $I(\phi)\subset\sO_X$. Cramer's rule tells us that this matrix annihilates the vector of the equations of $\Gamma$
\[ \begin{pmatrix} \: \phi & \begin{matrix} g \\ h \end{matrix} \end{pmatrix} \begin{pmatrix} \xi_2 \\ -\xi_1 \\ \eta\end{pmatrix} = 0. \]
Multiplying out these two matrices gives us two syzygies holding between the equations of $\Gamma$ and these syzygies define a variety 
\[ \sigma' \colon Y'\subset X\times \PP^2_{(\eta:\xi_1:\xi_2)} \to X. \] 
$Y'$ is the blowup of the (ordinary) power algebra $\bigoplus_{n\geq 0} \sI^n$ of the ideal $\sI=\sI_{\Gamma/X}$. 

$Y'$ cannot be the divisorial extraction of Theorem \ref{uniq} since the fibre above $P\in X$ is not 1-dimensional. Indeed $Y'$ contains a Weil divisor $\Pi=\sigma'^{-1}(P)_{\text{red}}\cong \PP^2$, possibly with a non reduced structure, defined by the ideal $I(\phi)$. Our aim is to construct the divisorial extraction by birationally contracting $\Pi$. This is done by unprojecting $I(\phi)$ and repeating this process for any other divisors that appear in the central fibre.

\begin{lem}
Suppose there exists a Mori extraction $\sigma\colon (E\subset Y) \to (\Gamma\subset X)$. Then at least one of $g,h$ is not in $\frakm\cdot I(\phi)$.
\label{Lem2}
\end{lem}
\begin{proof}
Suppose that both $g,h \in \frakm\cdot I(\phi)$. Then the three equations of $\Gamma$ satisfy $\eta\in I(\phi)^2$ and $\xi_1,\xi_2\in\frakm\cdot I(\phi)^2$. On the variety $Y$ there is a point $Q=Q_\eta\in Y$ in the fibre above $P$ where all variables except $\eta$ vanish. Now $x,y,z,\xi_1,\xi_2$ are all linearly independent elements of the Zariski tangent space $T_QY=(\frakm_Q/\frakm_Q^2)^\vee$. This $Q\in Y$ is a Gorenstein point with $\dim T_QY\geq5$, so $Q\in Y$ cannot be cDV and is therefore not terminal.
\end{proof}

This condition gives an upper bound on the multiplicity of $\Gamma$ at $P\in X$.

\section{Divisorial Extractions from Singular Curves: Type $A$}

In the absence of any kind of structure theorem for the general $\bA_n^j$ case, I give some examples to give a flavour of the kind of behaviour that occurs. As seen in other problems, for instance Mori's study of Type $A$ flips or the Type $A$ case of Tziolas' classification \cite{tz3}, this will be a big class of examples with lots of interesting and complicated behaviour. These varieties are described as serial unprojections, existing in arbitrarily large codimension, and look very similar to Brown and Reid's Diptych varieties \cite{dip} also constructed by serial unprojection.

The general strategy is to use Lemma \ref{Lem1} to write down the equations of the curve $\Gamma\subset X$, possibly using Lemma \ref{Lem2} and some extra tricks to place further restrictions on the functions $g,h$. Then, as described in \S\ref{1st_unproj}, we can take the unprojection plane $\Pi\subset Y'$ as our starting point and repeatedly unproject until we obtain a variety $\sigma\colon Y\to X$ with a small (i.e.\ 1-dimensional) fibre above $P$. This is the unique extraction described by Theorem \ref{uniq} so checking the singularities of $Y$ will establish the existence of a terminal extraction.

\subsection{Prokhorov and Reid's example} \
\label{prokreid}

I run through the easiest case in detail as an introduction to how these calculations work. This example first appeared in \cite{pr} Theorem 3.3 and a similar example appears in Takagi \cite{tak} Proposition 7.1.

Suppose that a general section $P\in\Gamma\subset S\subset X$ is of type $A_1$ (i.e.\ the case $\bA_1^1$ in the notation of Table \ref{mftab}). By Lemma \ref{Lem1} we are considering a curve $\Gamma\subset S\subset X$ given by the equations
\[ \bigwedge^2 \begin{pmatrix}
x & y & -g(y,z) \\
y & z & h(x,y)
\end{pmatrix} = 0 \]
where the minus sign is chosen for convenience and we can use column operations to eliminate any occurrence of $x$ (resp.\ $z$) from $g$ (resp.\ $h$). Moreover $g,h\in I(\phi)=\frakm$ so we can write $g=cy+dz$ and $h=ax+by$ for some choice of functions $a,b,c,d\in\sO_X$.

By Lemma \ref{Lem2} at least one of $a,b,c,d\not\in\frakm$ else the divisorial extraction is not terminal. This implies that $\Gamma$ has multiplicity three at $P$. If we consider $S$ as the quotient $\CC^2_{u,v}/\ZZ_2$, where $x,y,z=u^2,uv,v^2$, then $\Gamma$ is given by the orbifold equation
\[ \gamma(u,v) = au^3+bu^2v+cuv^2+dv^3   \]
and the tangent directions to the branches of $\Gamma$ at $P$ correspond to the three roots of this equation.

Recall \emph{Cramer's rule} in linear algebra: that any $n\times(n+1)$ matrix annihilates the associated vector of $n\times n$ minors. 

In our case this gives two syzygies between the equations of $\Gamma\subset X$
\begin{equation} \begin{pmatrix}
x & y & -(cy + dz) \\
y & z & ax + by 
\end{pmatrix} \begin{pmatrix}
\xi_2 \\ -\xi_1 \\ \eta
\end{pmatrix} = 0
\tag{$*$}\label{eqns}\end{equation}
where $\eta=xz-y^2$ is the equation of $S$ and $\xi_1,\xi_2$ are the other two equations of $\Gamma$. We can write down a codimension 2 variety 
\[ \sigma' \colon Y'\subset X\times\PP^2_{(\xi_1:\xi_2:\eta)} \to X\]
where $\sigma'$ is the natural map given by substituting the equations of $\Gamma$ back in for $\xi_1,\xi_2,\eta$. Outside of $P$ this map $\sigma'$ is isomorphic to the blowup of $\Gamma$, in fact $Y'$ is the blowup of the \emph{ordinary power algebra} $\bigoplus \sI^n$. However $Y'$ cannot be the unique divisorial extraction described in Theorem \ref{uniq} since the fibre over the point $P$ is not small. Indeed, $Y'$ contains the plane $\Pi:=\sigma'^{-1}(P )_{\text{red}}\cong \PP^2$.

We can unproject $\Pi$ by rewriting the equations of $Y'$ \eqref{eqns} so that they annihilate the ideal $(x,y,z)$ defining $\Pi$,
\[ \begin{pmatrix}
\xi_2 & \xi_1 + c\eta & -d\eta \\  
-a\eta & \xi_2 + b\eta & \xi_1 \\
\end{pmatrix} \begin{pmatrix}
x \\ -y \\ z \end{pmatrix} = 0. \]
By using Cramer's rule again, we see that $Y'$ has some nodal singularities along $\Pi$ where $x,y,z$ and the minors of this new $2\times 3$ matrix all vanish. If the roots of $\gamma$ are distinct then this locus consists of three ordinary nodal singularities along $\Pi$. If $\gamma$ acquires a double (resp.\ triple) root then two (resp.\ three) of these nodes combine to give a slightly worse nodal singularity. 

We can resolve these nodes by introducing a new variable $\zeta$ that acts as a ratio between these two vectors, i.e.\ $\zeta$ should be a degree 2 variable satisfying the three equations 
\begin{align*} x\zeta &= \xi_1(\xi_1+c\eta)+d(\xi_2+b\eta)\eta, \\
y\zeta &= \xi_1\xi_2 - ad\eta^2, \\
z\zeta &= \xi_2(\xi_2+b\eta) + a(\xi_1+c\eta)\eta. \end{align*}
This all gives a codimension 3 variety $\sigma\colon Y \subset X\times \PP(1,1,1,2) \to X$ defined by five equations. As described in \S\ref{unproj}, by the Buchsbaum--Eisenbud theorem we can write these equations neatly as the maximal Pfaffians of the skew-symmetric $5\times 5$ matrix
\[ \begin{pmatrix}
\zeta & \xi_2 & \xi_1 + c\eta & -d\eta \\  
& -a\eta & \xi_2 + b\eta & \xi_1 \\
& & z & y \\
& & & x
\end{pmatrix} \]
(where the diagonal of zeroes and antisymmetry are omitted for brevity).

Now we can check that $Y$ actually is the divisorial extraction from $\Gamma$. Outside of the central fibre $Y$ is still the blowup of $\Gamma$, since 
\[ Y\setminus \sigma^{-1}(P)\cong Y'\setminus \sigma'^{-1}(P). \] 
The plane $\Pi\subset Y'$ is contracted to the coordinate point $Q_\zeta\in Y$ where all variables except $\zeta$ vanish. ($Q_\zeta$ is called the \emph{unprojection point} of $Y$ since the map $Y\dashrightarrow Y'$ is projection from $Q_\zeta$.) The central fibre is the union of (at most) three lines, all meeting at $Q_\zeta\in Y$. Therefore $\sigma$ is small and, by Theorem \ref{uniq} on the uniqueness of contractions, this has to be the divisorial extraction from $\Gamma$.

Furthermore we can check that $Y$ is terminal. First consider an open neighbourhood of the unprojection point $(Q_\zeta\in U_\zeta):=\{\zeta=1\}$. We can eliminate the variables $x,y,z$ to see that this open set is isomorphic to the cyclic quotient singularity
\[ (Q_\zeta\in U_\zeta)\cong (0\in \CC^3_{\xi_1,\xi_2,\eta})/\tfrac12(1,1,1). \]
Now for each line $L\subseteq \sigma^{-1}(P)_{\text{red}}$ we are left to check the point $Q_L=L\cap\{\zeta=0\}$. Note that each of these points lies in the affine open set $U_\eta=\{\eta=1\}$ and recall that at least one of the coefficients $a,b,c,d$ is a unit. After a possible change of variables, we may assume $a\notin\frakm$. We can use the equations involving $a$ above to eliminate $x$ and $\xi_1$. After rewriting $\zeta=a\zeta', \xi_2=a\xi'_2$, we are left with the equation of a hypersurface
\[ \big((y - z\xi'_2)\zeta' + a{\xi'_2}^3 + b{\xi'_2}^2 + c\xi'_2 + d = 0\big) \subset \Aa^4_{y,z,\xi'_2,\zeta'} \]
which is smooth (resp.\ $cA_1,cA_2$) at $Q_L$ if $L$ is the line over a node corresponding to a unique (resp.\ double, triple) root of $\gamma$. 

If we consider the case where all of $a,b,c,d\in\frakm$ then the central fibre consists of just one line $L$ and the point $Q_L\in Y$ is not terminal (the matrix defining $Y$ has rank 0 at this point, so it cannot be a hyperquotient) which agrees with Lemma \ref{Lem2}.

\begin{rmk}\label{geom}
The following construction, originally due to Hironaka, illustrates how the unprojection of $\Pi$ works geometrically.

Consider the variety $X'$ obtained by the blowup of $P\in X$ followed by the blowup of the birational transform of $\Gamma$. The exceptional locus has two components $\Pi_{X'}$ and $E_{X'}$ dominating $P$ and $\Gamma$ respectively. Assuming the tangent directions of the branches of $\Gamma$ at $P$ are distinct then $\Pi_{X'}$ is a Del Pezzo surface of degree 6. Consider the three $-1$-curves of $\Pi_{X'}$ that don't lie in the intersection $\Pi_{X'}\cap E_{X'}$. They have normal bundle $\sO_{X'}(-1,-1)$ so we can flop them. The variety $Y'$, constructed above, is the midpoint of this flop and we end up with the following diagram,
\begin{center}
\begin{tikzpicture}
  \node at (0,0) {$X$};
  \node at (1,1) {$X'$};
  \node at (2,0) {$Y'$};
  \node at (3,1) {$Z$};
  \node at (4,0) {$Y$};
  
  \path[->] (.75,.75) edge (.25,.25); 
  \path[->] (1.25,.75) edge (1.75,.25);
  \path[->] (2.75,.75) edge (2.25,.25); 
  \path[->] (3.25,.75) edge (3.75,.25);
  \path[dashed,->] (1.5,1) edge node[above] {\tiny{flop}} (2.75,1);
\end{tikzpicture}.
\end{center}
The plane $\Pi\subset Y$ is the image of $\Pi_{X'}$ with the three nodes given by the contracted curves. After the flop the divisor $\Pi_{X'}$ becomes a plane $\Pi_Z\cong \PP^2$ with normal bundle $\sO_Z(-2)$, so we can contract it to get $Y$ with a $\tfrac12$-quotient singularity. If we want to consider non-distinct tangent directions then this picture becomes more complicated.
\end{rmk}

\begin{rmk}
Looking back at \eqref{eqns} one may ask what happens if we unproject the ideal $(\xi_1,\xi_2,\eta)\subset \sO_{Y'}$ or, equivalently, the Jer$_{12}$ ideal $(\xi_1,\xi_2,\eta,\zeta)\subset\sO_Y$. Even though this may not appear to make sense geometrically, it is a well-defined operation in algebra. If we do then we introduce a variable $\iota$ of weight $-1$ that is nothing other than the inclusion $\iota \colon \sI_\Gamma \hookrightarrow \sO_X$. The whole picture is a big graded ring
\[ \sR := \sO_X(-1,1,1,1,2) / (\text{codim 4 ideal})  \]
and writing $\sR_+$ (resp.\ $\sR_-$) for the positively (resp.\ negatively) graded part of $\sR$ we can construct the divisorial extraction in the style of \cite{whatis}, as the Proj of a $\ZZ$-graded algebra
\begin{center}
\begin{tikzpicture}
  \node at (0.5,1) {$\Proj_{\sO_X} \sR_-$};
  \node at (3,0) {$X= \Spec \sO_X$};
  \node at (6,1) {$Y = \Proj_{\sO_X} \sR_+$};
  
  \draw[double, double distance = 2pt] (1.5,.8) -- (2,.3); 
  \path[->] (4.5,.8) edge node [right] {$\sigma$} (4,.3);
\end{tikzpicture}.
\end{center}
\end{rmk}

\begin{rmk}
The unprojection variable $\zeta$ corresponds to a generator of $\bigoplus\sI^{[n]}$ that lies in $\sI^{[2]}\setminus \sI^2$. Either by writing out one of the equations involving $\zeta$ and substituting for the values of $\xi_1,\xi_2,\eta$, or by calculating the unprojection equations of $\iota$, we can give an explicit expression for $\zeta$ as
\[ \iota\zeta = (ax+by)\xi_1 + (cy+dz)\xi_2 + (acx+ady+bcy+bdz)\eta. \]
In terms of the orbifold equation $\gamma$, the generators $\xi_1,\xi_2,\zeta$ are lifts modulo $\eta$ of the forms $u\gamma,v\gamma,\gamma^2$ defined on $S$.
\end{rmk}

\subsection{The $\bA_2^1$ case}\

Suppose that the general section $P\in\Gamma\subset S\subset X$ is of type $\bA_2^1$. By Lemma \ref{Lem1}, we are considering the curve given by the equations
\[ \bigwedge^2 \begin{pmatrix}
x & y & -(dy+ez) \\
y^2 & z & ax+by
\end{pmatrix} = 0 \]
for some choice of functions $a,b,d,e\in\sO_X$. If $a,b,d,e$ are taken generically then the general section through $\Gamma$ is of type $A_1$, so we need to introduce some more conditions on these functions. 

Consider the section $H_{\lambda,\mu} = \{ h_{\lambda,\mu} := \eta+\lambda\xi_1+\mu\xi_2 =0\}$.  The quadratic term of this equation is given by
\[ h^{(2)}_{\lambda,\mu} = xz + \lambda x(a_0x+b_0y) + \mu (a_0xy+b_0y^2+d_0yz+e_0z^2) \]
where $a_0$ is the constant term of $a$ and similarly for $b,d,e$. To ensure the general section is of type $A_2$ it is enough to ask that $h_{\lambda,\mu}^{(2)}$ has rank 2 for all $\lambda,\mu$. After playing around, completing the square etc., we get two cases according to whether $x\mid h_{\lambda,\mu}^{(2)}$ or $z\mid h_{\lambda,\mu}^{(2)}$:
\begin{align*}
a_0=b_0=0 &\implies h_{\lambda,\mu}^{(2)} = z(x + \mu d_0y + \mu e_0z), \\
b_0=d_0=e_0=0 &\implies  h_{\lambda,\mu}^{(2)} = x(z + \lambda a_0x + \mu a_0y).
\end{align*}

\subsubsection{Case 1---Tom$_1$}\

Take the first case where $a_0=b_0=0$. Then we can rewrite $ax+by$ as $ax^2+bxy+cy^2$, so that the equations of $\Gamma$ become
\[ \bigwedge^2 \begin{pmatrix}
x & y & -\left(dy+ez\right) \\
y^2 & z & ax^2+bxy+cy^2
\end{pmatrix} = 0. \]
\emph{Claim:} The following two conditions must hold
\begin{enumerate}
\item one of $a,b,c,d\notin\frakm$,
\item one of $d,e\notin\frakm$,
\end{enumerate}
and (after possibly changing variables) we can assume that $a,e\notin\frakm$.

Statement (2) follows from Lemma \ref{Lem2}. The first is also proved in a similar way. If (1) does not hold then necessarily $e\notin\frakm$ by (2). Consider the point $Q_\eta\in Y$ where all variables but $\eta$ vanish, as in the proof of Lemma \ref{Lem2}. This is a Gorenstein point with local equation
\[ ey^2\xi_1 - x\xi_1\xi_2 + y\xi_2^2 + dy\xi_2 + e(ax^2+bxy + cy^2) = 0 \]
and if $a,b,c,d\in\frakm$ then this equation is not cDV as no terms of degree $2$ appear, so it is not terminal.

By considering the minimal resolution $\widetilde{S}\to S$, we see that any $\Gamma$ that satisfies these conditions is a curve whose birational transform $\widetilde{\Gamma}\subset\widetilde{S}$ intersects the exceptional locus with multiplicities 
\begin{center} \begin{tikzpicture}
	\node[label={[label distance=-0.2cm]90:\tiny{$3$}}] at (1,1) {$\bullet$};
	\node[label={[label distance=-0.2cm]90:\tiny{$1$}}] at (2,1) {$\bullet$};
	\draw
       		(1.075,1)--(1.925,1);
\end{tikzpicture},  \end{center} 
i.e.\ $\widetilde{\Gamma}$ intersects $E_1=\PP^1_{(x_1:x_2)}$ with multiplicity three and $E_2=\PP^1_{(y_1:y_2)}$ with multiplicity one, according to the (nonzero) equations 
\begin{align*} 
\widetilde{\Gamma}\cap E_1 &\colon \quad a_0x_1^3 + b_0x_1^2x_2 + c_0x_1x_2^2 + d_0x_2^3 = 0, \\
\widetilde{\Gamma}\cap E_2 &\colon \quad d_0 y_1 + e_0 y_2 = 0. 
\end{align*}

If we follow the Prokhorov--Reid example, we can write down a codimension 3 model of the blowup of $\Gamma$ as $\sigma'' \colon Y''\subset X\times\PP(1,1,1,2) \to X$ given by the the Pfaffians of the matrix
\[ \begin{pmatrix}
\zeta & \xi_2 & \xi_1 + d\eta & -e\eta \\ 
& -(ax+by)\eta & y(\xi_2 + c\eta) & \xi_1 \\
& & z & y \\
& & & x
\end{pmatrix} \]

The variety $Y''$ is \emph{not} the divisorial extraction since $\sigma''$ is not small. A new unprojection plane appears after the first unprojection. This plane $\Pi$ is defined by the ideal $(x,y,z,\xi_1)$ and we can see that the matrix is in Tom$_1$ format with respect to this ideal. The central fibre $\sigma''^{-1}(P)$ is given by $\Pi$ together with the line 
\[L_1=(x=y=z=\xi_2=\xi_1+d\eta=0).\]

Unprojecting $\Pi$ gives a new variable $\theta$ of weight three with four additional equations
\begin{align*} 
x\theta &= (\zeta + be\eta^2)(\xi_1+d\eta) + e\xi_2(\xi_2+c\eta)\eta \\
y\theta &= \xi_2\zeta - ae(\xi_1+d\eta)\eta^2 \\
z\theta &= \xi_2^2(\xi_2+c\eta) + b\xi_2(\xi_1+d\eta)\eta + a(\xi_1+d\eta)^2\eta \\
\xi_1\theta &= \zeta(\zeta + be\eta^2) + ae^2(\xi_2+c\eta)\eta^3
\end{align*}

Generically, the central fibre consists of four lines passing through the point $P_\theta$, the line $L_1$ and the three lines that appear after unprojecting $\Pi$. The open neighbourhood $(P_\theta\in U_\theta)$ is isomorphic to a $\tfrac13(1,1,2)$ singularity. As we assume $a,e\notin\frakm$, when $\eta=1$ we can use the equations to eliminate $x,z,\xi_1,\xi_2$ so that all the points $Q_L=L\cap\{\zeta=0\}$, for $L\subseteq \sigma^{-1}(P)_{\text{red}}$, are smooth.

\subsubsection{Case 2---Jer$_{45}$}\

Now consider instead the case where $b_0=d_0=e_0=0$. In direct analogy to the Tom$_1$ case the reader can check that
\begin{enumerate} 
\item $\Gamma$ is a curve of type
\begin{center} \begin{tikzpicture}
	\node[label={[label distance=-0.2cm]90:\tiny{$4$}}] at (1,1) {$\bullet$};
	\node at (2,1) {$\bullet$};
	\draw
       		(1.075,1)--(2,1);
\end{tikzpicture},  \end{center}
\item after making the first unprojection we get a variety $Y'$ containing a plane $\Pi$ above $P$ defined by the Jer$_{45}$ ideal $(x,y,z,\xi_2)$,
\item $Y'$ has (at most) four nodes along $\Pi$ corresponding to the roots of the orbifold equation $\gamma$, 
\item after unprojecting $\Pi$ we get a variety $Y$ with small fibre over $P$, hence $Y$ is the divisorial extraction,
\item the open neighbourhood of the unprojection point $(P_\theta\in U_\theta)$ is isomorphic to the quotient singularity $\frac13(1,1,2)$,
\item $Y$ has at worst $cA$ singularities at the points $Q_L$ according to whether $\gamma$ has repeated roots.
\end{enumerate}

\subsection{An $\bA_3^2$ example} \

Suppose that the general section $P\in\Gamma\subset S\subset X$ is of type $\bA_3^2$ and that $\Gamma$ is a curve whose birational transform on a resolution of $S$ intersects the exceptional divisor with multiplicities
\begin{center} \begin{tikzpicture}
	\node at (1,1) {$\bullet$};
	\node[label={[label distance=-0.2cm]90:\tiny{$3$}}] at (2,1) {$\bullet$};
	\node at (3,1) {$\bullet$};
	\draw
       		(1,1)--(1.925,1)
       		(2.075,1)--(3,1);
\end{tikzpicture}.  \end{center}
Then a terminal extraction from $\Gamma\subset X$ exists.

The calculation is very similar to the Prokhorov--Reid example, except that the first unprojection divisor $\Pi\subset Y'$ is defined by the ideal $I(\phi)=(x,y^2,z)$, so that $\Pi$ is \emph{not reduced}. After unprojecting $\Pi$ we get an index 2 model $Y\subset X\times \PP(1,1,1,2)$ for the divisorial extraction with equations 
\[ \begin{pmatrix}
\zeta & \xi_2 & \xi_1 + c\eta & -d\eta \\ 
& -a\eta & \xi_2 + b\eta & \xi_1 \\
& & z & y^2 \\
& & & x 
\end{pmatrix} \]
$\Pi$ is contracted to a singularity of type $cA_1/2$, given by the $\tfrac12$-quotient of the hypersurface singularity
\[ y^2 - \xi_1\xi_2 + ad\eta^2 = 0. \]

\section{Divisorial Extractions from Singular Curves: Types $D$ \& $E$}

The result of the calculations in this section are summed up in the following theorem.
\begin{thm} \label{excthm}
Suppose $P\in \Gamma\subset S\subset X$ as in \S\ref{geconj}. 
\begin{enumerate}
\item Suppose that $\Gamma$ is of type $\bD_n^l,\bD_{2k}^r$ or $\bE_7$. Then the divisorial extraction has a codimension 3 model 
\[\sigma \colon Y\subset X\times \PP(1,1,1,2) \to X \] 
In particular, $Y$ has index 2 and $\bigoplus\sI^{[n]}$ is generated in degrees $\leq 2$. 

Moreover, $Y$ is singular along a component line of the central fibre, so there does not exist a terminal extraction from $\Gamma$. 
\item Suppose that $\Gamma$ is of type  $\bE_6$. We need to consider two cases.
\begin{enumerate}
\item The restriction map $\sI_{\Gamma/X}\to\sI_{\Gamma/S}$ is surjective. Then the divisorial extraction has a codimension 4 model 
\[\sigma \colon Y\subset X\times \PP(1,1,1,2,3) \to X \] 
In particular, $Y$ has index 3 and $\bigoplus\sI^{[n]}$ is generated in degrees $\leq 3$. 

Moreover, if $Y$ is terminal then $\Gamma$ is a curve of type 
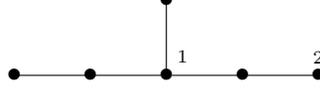
\begin{center} \begin{tikzpicture}
	\node at (1,1) {$\bullet$};
	\node at (2,1) {$\bullet$};
	\node[label={[label distance=-0.2cm]80:\tiny{$1$}}] at (3,1) {$\bullet$};
	\node at (4,1) {$\bullet$};
	\node[label={[label distance=-0.2cm]90:\tiny{$2$}}] at (5,1) {$\bullet$};
	\node at (3,2) {$\bullet$};
	\draw
       		(1,1)--(2.925,1)
       		(3.075,1)--(4.925,1)
       		(3,1.075)--(3,2);  
\end{tikzpicture}  \end{center}

\item The restriction map $\sI_{\Gamma/X}\to\sI_{\Gamma/S}$ is not surjective. Then the divisorial extraction has a codimension 5 model 
\[\sigma \colon Y\subset X\times \PP(1,1,1,2,3,4) \to X \] 
In particular, $Y$ has index 4 and $\bigoplus\sI^{[n]}$ is generated in degrees $\leq 4$.  

Moreover, if $Y$ is terminal then $\Gamma$ is a curve of type
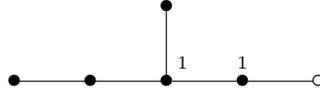
\begin{center} \begin{tikzpicture}
	\node at (1,1) {$\bullet$};
	\node at (2,1) {$\bullet$};
	\node[label={[label distance=-0.2cm]80:\tiny{$1$}}] at (3,1) {$\bullet$};
	\node[label={[label distance=-0.2cm]90:\tiny{$1$}}] at (4,1) {$\bullet$};
	\node at (5,1) {$\circ$};
	\node at (3,2) {$\bullet$};
	\draw
       		(1,1)--(2.925,1)
       		(3.075,1)--(3.925,1)
       		(4.075,1)--(4.925,1)
       		(3,1.075)--(3,2);  
\end{tikzpicture}  \end{center}
In this case, the central fibre $Z\subset Y$ is a union of lines meeting at a $cAx/4$ singularity. The curve marked $\circ$ is pulled out in a partial resolution of $S$.
\end{enumerate}
\end{enumerate}
\end{thm}

Before launching into the proof of this theorem note the following useful remark.
\begin{rmk} \label{msqu}
Suppose the general section $P\in\Gamma\subset S\subset X$ is of type $D$ or $E$. Then we can write the equations of $\Gamma$ as
\[ \bigwedge^2 \begin{pmatrix} \: \phi & \begin{matrix} -g(y,z) \\ h(y,z) \end{matrix} \end{pmatrix} = 0 \]
where $g,h\in \frakm^2\cap I(\phi)$. To see this consider the matrix factorisations in Table \ref{mftab}. Firstly, we can use column operations to cancel any terms involving $x$ from $g,h$. Then to prove $g,h\in\frakm^2$ consider the section $h_{\lambda,\mu}=\eta + \lambda\xi_1 + \mu\xi_2$. The quadratic term of $h_{\lambda,\mu}$ is
\[ h_{\lambda,\mu}^{(2)} = x^2 + \lambda xh^{(1)} + \mu xg^{(1)} + \lambda tg^{(1)} \quad \text{(where $t=y$ or $z$)} \]
and we require this to be a square for all $\lambda,\mu$. This happens only if $g^{(1)}=h^{(1)}=0$.
\end{rmk}

\subsection{The $\bD_n^l,\bD_{2k}^r$ and $\bE_7$ cases} \

These three calculations are essentially all the same. Since they are so similar we only do the $\bD_n^l$ case explicitly.

\subsubsection{The $\bD_n^l$ case}\

According to Lemma \ref{Lem1} and Remark \ref{msqu}, the curve $\Gamma\subset S\subset X$ is defined by the equations 
\[ \bigwedge^2 \begin{pmatrix}
x & y^2+z^{n-2} & a(y^2+z^{n-2}) + byz + cz^2 \\
z & x & d(y^2+z^{n-2}) + eyz + fz^2 \\
\end{pmatrix} = 0 \]
for some functions $a,b,c,d,e,f\in\sO_X$. Unprojecting $I(\phi)$ gives a variety 
\[ \sigma \colon Y\subset X\times \PP(1,1,1,2) \to X\]
with equations given by the maximal Pfaffians of the matrix
\[ \begin{pmatrix}
\zeta & \xi_2 & \xi_1 - a\eta & (by+cz)\eta \\
& -\xi_1 & -d\eta & \xi_2 + (ey+fz)\eta \\
& & z & y^2+z^{n-2} \\
& & & x
\end{pmatrix}.  \]
This $\sigma$ is a small map, so that $Y$ is the divisorial extraction of $\Gamma$. Indeed, the central fibre $Z=\sigma^{-1}(P)_{\text{red}}$ consists of two components meeting at the point $P_\zeta$. These are the lines 
\begin{align*} 
L_1 &=(x=y=z=\xi_1=\xi_2=0)\\ 
L_2 &=(x=y=z=\xi_1-a\eta=\xi_2=0)
\end{align*}

Looking at the affine patch $U_\zeta := \{\zeta=1\}\subset Y$ we see that we can eliminate the variables $x,z$ and that $U_\zeta$ is a $\tfrac12$-quotient of the hypersurface singularity
\[ y^2 + z^{n-2} - \xi_2^2 - (e\xi_2+b\xi_1)y\eta - (f\xi_2+c\xi_1)z\eta = 0 \]
where $z= \xi_1^2 - (a\xi_1 + d\xi_2)\eta$. 

This hypersurface is singular along the line $L_1$ since this equation is contained in the square of the ideal $(y,\xi_1,\xi_2)$. Therefore $Y$ has nonisolated singularities and cannot be terminal.

\subsection{The $\bE_6$ case} \

Suppose that $\Gamma\subset S\subset X$ is of type $\bE_6$. By Lemma \ref{Lem1} the equations of $\Gamma$ can be written in the form
\[\bigwedge^2\begin{pmatrix}
x & y^2 & -g(y,z) \\
y & x+z^2 & h(y,z) \\
\end{pmatrix}=0\]
where $g,h\in\frakm^2$ by Remark \ref{msqu}. Now consider the general section $H_{\lambda,\mu} = \eta+\lambda\xi_1+\mu\xi_2$. After making the replacement $x \mapsto x+\tfrac12(\lambda h + \mu g)$ the cubic term of $H_{\lambda,\mu}$ is given by
\[ x^2 - y^3 + \lambda yg^{(2)} \]
where $g^{(2)}$ is the 2-jet of $g$. For the general $H_{\lambda,\mu}$ to be of type $E_6$, we require $y(y^2 - \lambda g^{(2)})$ to be a perfect cube for all values of $\lambda$. This happens only if $g^{(2)}$ is a multiple of $y^2$. Therefore we can take $g$ and $h$ to be
\[ g(y,z) = a(y,z)y^2 + b(z)yz^2 + c(z)z^3, \quad h(y,z) = d(y)y^2 + e(y)yz + f(y,z)z^2, \]
for some choice of functions $a,b,c,d,e,f\in\sO_X$. Moreover, $f\not\in\frakm$ else the extraction is not terminal by Lemma \ref{Lem2}.

By specialising these coefficients the curve we are considering varies. After writing down the minimal resolution $\widetilde{S}\to S$ explicitly, one can check that the birational transform of $\Gamma$ is a curve intersecting the exceptional locus of $\widetilde{S}$ with the following multiplicities:
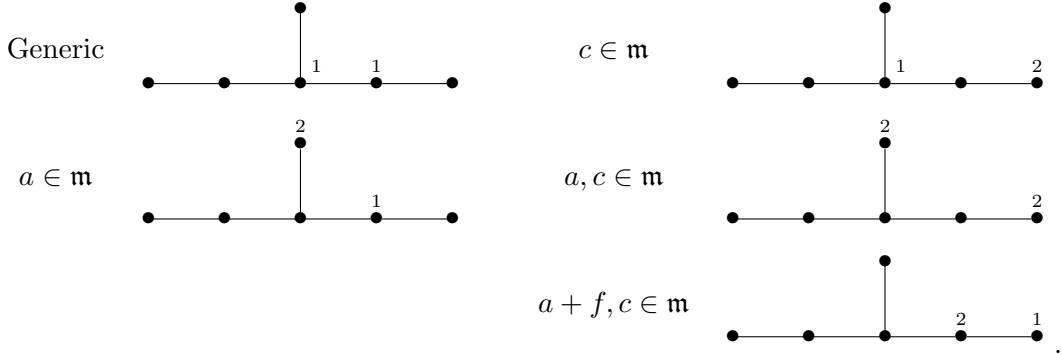
\begin{center}
\begin{tabular}{cm{5cm}cm{5cm}}
Generic & \begin{tikzpicture}
	\node at (1,1) {$\bullet$};
	\node at (2,1) {$\bullet$};
	\node[label={[label distance=-0.2cm]80:\tiny{$1$}}] at (3,1) {$\bullet$};
	\node[label={[label distance=-0.2cm]90:\tiny{$1$}}] at (4,1) {$\bullet$};
	\node at (5,1) {$\bullet$};
	\node at (3,2) {$\bullet$};
	\draw
       		(1,1)--(2.925,1)
       		(3.075,1)--(3.925,1)
       		(4.075,1)--(4.925,1)
       		(3,1.075)--(3,2);  
\end{tikzpicture} & $c\in\frakm$ & \begin{tikzpicture}
	\node at (1,1) {$\bullet$};
	\node at (2,1) {$\bullet$};
	\node[label={[label distance=-0.2cm]80:\tiny{$1$}}] at (3,1) {$\bullet$};
	\node at (4,1) {$\bullet$};
	\node[label={[label distance=-0.2cm]90:\tiny{$2$}}] at (5,1) {$\bullet$};
	\node at (3,2) {$\bullet$};
	\draw
       		(1,1)--(2.925,1)
       		(3.075,1)--(4.925,1)
       		(3,1.075)--(3,2);  
\end{tikzpicture} \\ 

$a\in\frakm$ & \begin{tikzpicture}
	\node at (1,1) {$\bullet$};
	\node at (2,1) {$\bullet$};
	\node at (3,1) {$\bullet$};
	\node[label={[label distance=-0.2cm]90:\tiny{$1$}}] at (4,1) {$\bullet$};
	\node at (5,1) {$\bullet$};
	\node[label={[label distance=-0.2cm]90:\tiny{$2$}}] at (3,2) {$\bullet$};
	\draw
       		(1,1)--(3.925,1)
       		(4.075,1)--(4.925,1)
       		(3,1)--(3,1.925);  
\end{tikzpicture} & $a,c\in\frakm$ & \begin{tikzpicture}
	\node at (1,1) {$\bullet$};
	\node at (2,1) {$\bullet$};
	\node at (3,1) {$\bullet$};
	\node at (4,1) {$\bullet$};
	\node[label={[label distance=-0.2cm]90:\tiny{$2$}}] at (5,1) {$\bullet$};
	\node[label={[label distance=-0.2cm]90:\tiny{$2$}}] at (3,2) {$\bullet$};
	\draw
       		(1,1)--(4.925,1)
       		(3,1)--(3,1.925);  
\end{tikzpicture} \\

& & $a+f,c\in\frakm$ & \begin{tikzpicture}
	\node at (1,1) {$\bullet$};
	\node at (2,1) {$\bullet$};
	\node at (3,1) {$\bullet$};
	\node[label={[label distance=-0.2cm]90:\tiny{$2$}}] at (4,1) {$\bullet$};
	\node[label={[label distance=-0.2cm]90:\tiny{$1$}}] at (5,1) {$\bullet$};
	\node at (3,2) {$\bullet$};
	\draw
       		(1,1)--(3.925,1)
       		(4.075,1)--(4.925,1)
       		(3,1)--(3,2);  
\end{tikzpicture}. \\
\end{tabular}
\end{center}

The first unprojection $\sigma'\colon Y' \to X$ is defined by the maximal Pfaffians of the matrix 
\begin{equation}\begin{pmatrix}
\zeta & \xi_2 & y(\xi_1+a\eta) & -(by+cz)\eta \\
 & \xi_1 & \xi_2+(dy+ez)\eta & \xi_1-f\eta \\
 & & z^2 & y \\
 & & & x 
\end{pmatrix} \label{e6eqns}\tag{$\dagger$}\end{equation}
and this $Y'$ contains a new unprojection divisor defined by an ideal $I$ in Tom$_2$ format. If the coefficient $c$ is assumed to be chosen generally then $I=(x,y,z,\xi_2)$. However, if we make the specialisation $c\in\frakm$, we can take $I$ to be the smaller ideal $(x,y,z^2,\xi_2)$. Unprojecting these two ideals gives very different varieties.

\subsubsection{The special $\bE_6$ case: $c\in\frakm$}\

Since it is easier, consider first the case when $c\in\frakm$, i.e.\ we let $c(z)=c'(z)z$. Unprojecting $(x,y,z^2,\xi_2)$ gives a codimension 4 model,
\[ \sigma \colon Y \subset X\times \PP(1,1,1,2,3) \to X \] 
defined by the five Pfaffians above \eqref{e6eqns}, plus four additional equations:
\begin{align*}
x\theta &= (\xi_1+a\eta)(\xi_1-f\eta)^2 + b(\xi_1-f\eta)(\xi_2+(dy+ez)\eta)\eta +c'(\xi_2+(dy+ez)\eta)^2\eta, \\
y\theta &= \zeta(\xi_1-f\eta) + c'\xi_1(\xi_2+(dy+ez)\eta)\eta, \\
z^2\theta &= (\zeta-b\xi_1\eta)(\xi_2+(dy+ez)\eta) - \xi_1(\xi_1+a\eta)(\xi_1-f\eta), \\
\xi_2\theta &= \zeta(\zeta - b\xi_1\eta) + c'\xi_1^2(\xi_1+a\eta)\eta.
\end{align*}

The central fibre $Z$ is a union of three lines meeting at the unprojection point $P_\theta$, so that $Y$ is the divisorial extraction of $\Gamma$. These three lines are given by $x,y,z,\xi_2=0$ and
\begin{center}
\begin{tikzpicture}
\node at (0,1.2) {$L_1$};
\node at (0,0.6) {$L_2$};
\node at (0,0) {$L_3$};
\draw [decorate,decoration={brace,amplitude=3pt},xshift=-4pt,yshift=0pt]
(0.5,1.2) -- (0.5,0.6);
\node at (5,0.9) {$\xi_1-f\eta=\zeta^2 - bf\zeta\eta^2 + c'f^2(a+f)\eta^4=0$};
\node at (5,0) {$\xi_1+a\eta=\zeta=0$};
\end{tikzpicture}
\end{center}

In the open neighbourhood $P_\theta\in U_\theta$ we can eliminate $x,y,\xi_2$ by the equations involving $\theta$ above. We are left with a $\tfrac13$-quotient of the hypersurface singularity 
\[ H : z^2 = (\zeta-b\xi_1\eta)(\xi_2 + (dy+ez)\eta) - \xi_1(\xi_1+a\eta)(\xi_1-f\eta) \]

If $H$ is not isolated then $Y$ will have nonisolated singularities and there will be no terminal extraction from $\Gamma$. This happens if either $a\in\frakm$ or $a+f\in\frakm$. If $a\in\frakm$ then $H$ becomes singular along $L_3$. If $a+f\in\frakm$ then one of $L_1,L_2$ satisfies $\zeta - bf\eta^2=0$ and $H$ becomes singular along this line.

Now we can assume that $a,a+f,f\not\in\frakm$, and consider the (general) hyperplane section $\eta=0$, to see that $P_\theta\in U_\theta$ is the $cD_4/3$ point 
\[ \big( z^2 - \zeta^3 + \xi_1^3 + \eta(\cdots) = 0 \big) \: / \: \tfrac13(0,2,1,1;0). \]

\subsubsection{The general $\bE_6$ case: $c\not\in\frakm$}\

Now consider the more general case where $c$ is invertible. The difference between this and the last case is the existence of a form $\theta'$, vanishing three times on $\Gamma\subset S$, which fails to lift to $X$. 

We need to make two unprojections in order to construct the divisorial extraction $Y$. The first unprojection divisor defined by the Tom$_2$ ideal $(x,y,z,\xi_2)$ as described above. Then a new divisor appears defined by the ideal $\big(x,y,z,\xi_2,\xi_1(\xi_1+a\eta)\big)$. We add two new variables $\theta,\kappa$ of degrees 3,4 (resp.) to our ring and we end up with a variety in codimension 5
\[ \sigma \colon Y\subset X\times \PP(1,1,1,2,3,4) \to X. \]
The equations of $Y$ are given by the five equations \eqref{e6eqns} and nine new unprojection equations: four involving $\theta$ and five involving $\kappa$. The important equation is
\[ \xi_1(\xi_1+a\eta)\kappa = \zeta(\zeta - b\xi_1\eta)^2 - \theta(\theta - cd\xi_1\eta^2) + e\theta(\zeta-b\xi_1\eta)\eta + d\zeta(\xi_1-f\eta)(\zeta-b\xi_1\eta)\eta. \]
The open set of the unprojection point $P_\kappa\in U_\kappa$ is a hyperquotient point
\[ \xi_1^2 + \theta^2 - \zeta^3 + \eta(\cdots) = 0 \big) \: / \: \tfrac14(1,2,3,1;2), \]
which is the equation of a $cAx/4$ singularity. Moreover, one can check that this singularity is not isolated if $a\in\frakm$. Therefore, if $Y$ is terminal then $a\not\in\frakm$ and $\Gamma$ is as described in Theorem \ref{excthm}.

The central fibre of this extraction consists of (one or) two rational curves. One of these curves is pulled out in a partial resolution of $S$.

\subsection{The $\bD_{2k+1}^r$ case}\

This is certainly the most complicated of the exceptional cases and I intend to treat it fully in another paper, however some calculations predict that it should be similar to the $\bE_6$ case in the following sense. 

In general the restriction map $\sI_{\Gamma/X}\to \sI_{\Gamma/S}$ is not surjective, although after specialising some coefficients there is a good case where it does become surjective. 

If the map is not surjective then the divisorial extraction $\sigma\colon Y\to X$ pulls one or more curves out of $S$, so that pulling back to $S_Y\subset Y$ gives a partial crepant resolution $\sigma\colon S_Y\to S$.

In the good case we can unproject just three divisors, given by the chain of ideals
\[ (x,y,z^k), \quad (x,y,z,\xi_2), \quad (x,y,z,\xi_2,\xi_1^2)  \] 
to get a codimension 5 model $Y\subset X\times \PP(1,1,1,2,3,4)$ of index 4. The restriction to $\sigma\colon S_Y\to S$ is an isomorphism and, if it is isolated, the last unprojection point $Q\in S_Y\subset Y$ is a singularity of type $cAx/4$.

\end{document}